\documentclass{proc-l}

\usepackage{amssymb}

\usepackage{graphicx}

\usepackage{epstopdf}

\usepackage{stmaryrd}

\usepackage[utf8]{inputenc}

\newtheorem{theorem}{Theorem}[section]
\newtheorem{lemma}[theorem]{Lemma}
\newtheorem{proposition}[theorem]{Proposition}
\newtheorem{corollary}[theorem]{Corollary}

\theoremstyle{definition}
\newtheorem{definition}[theorem]{Definition}

\theoremstyle{remark}

\numberwithin{equation}{section}

\newcommand{\Ho}{\mathcal{H}}

\newcommand{\R}{\mathbb{R}}

\newcommand{\Z}{\mathbb{Z}}
\newcommand{\N}{\mathbb{N}}

\newcommand{\Cf}{\mathcal{C}}
\newcommand{\Ff}{\mathcal{F}}
\newcommand{\Hf}{\mathcal{H}}

\newcommand{\Uf}{\mathcal{U}}

\newcommand{\Df}{\mathcal{D}}

\begin{document}

\title[Generic shadowing property for conservative homeomorphisms]{On the genericity of the shadowing property for conservative homeomorphisms}


\author{Pierre-Antoine Guihéneuf}
\address{Universidade Federal Fluminense, Rua Mário Santos Braga, 24020-140, Niterói, RJ, Brasil}
\curraddr{}
\email{pguiheneuf@id.uff.br}
\thanks{}

\author{Thibault Lefeuvre}
\address{\'Ecole Polytechnique, Route de Saclay, 91128 Palaiseau, France}
\curraddr{}
\email{thibault.lefeuvre@polytechnique.org}
\thanks{}

\date{}

\dedicatory{}

\begin{abstract}
We prove the genericity of the shadowing and periodic shadowing properties for both conservative and dissipative homeomorphisms on a compact connected manifold. Our proof is valid for topological manifolds and still holds in the dissipative case. As a consequence of this result, we establish the genericity of the specification property, the average shadowing property and the asymptotic average shadowing property in the conservative case.
\end{abstract}

\maketitle


\section{Introduction}

Most of practical dynamical systems are very complex and subject to exterior perturbations, thus limiting their modelling to a relatively rough approximation. Therefore, one can wonder if the small discrepancy between the real system and its model has big consequences from a dynamical viewpoint? It turns out that for dynamical systems possessing the \emph{shadowing property} (see Definition~\ref{def:shadow}), the errors induced by the model do not destroy completely the dynamical behaviour: any orbit of the model is in fact close to some real orbit. In particular, this remark is still valid in the case where the ``real system'' is some abstract dynamics, and the model is a numerical simulation of it. The goal of this paper is to prove that ``most of'' the dynamical systems satisfy this shadowing property: both generic conservative and dissipative homeomorphisms of compact manifolds possess this property (see definitions below). Thus, most of the time, in some quite weak sense, both approximated models and numerical simulations are dynamically relevant.

\subsection{General set-up}

Throughout this paper, we will consider $M$, a compact connected manifold with or without boundary, of dimension $n \geq 2$, endowed with a distance $\text{dist}$. A \textit{good}\footnote{Also called OU (Oxtoby-Ulam) measure or Lebesgue-like measure in literature.} Borel probability measure $\mu$ on $M$ is a measure satisfying:
\begin{enumerate}
\item  $\forall p \in M, \mu(\left\{ p \right\} ) = 0$ (non-atomic);
\item  $\forall \Omega \subset M$ non-empty open set, $\mu(\Omega) > 0$ (full support);
\item  $\mu(\partial M) = 0$ (zero on the boundary).
\end{enumerate}
Once for all, we fix such a measure $\mu$. We denote by $\mathcal{H}(M)$ the set of homeomorphisms of $M$ and by $\mathcal{H}(M,\mu)$ the set of conservative homeomorphisms of $M$, namely the homeomorphisms on $M$ preserving the measure $\mu$. In the sequel, $\Ho$ will denote both spaces $\mathcal{H}(M)$ and $\mathcal{H}(M,\mu)$. The spaces $\Ho$ are metrizable by the $\mathcal{C}^0$ distance $d(f,g) = \max_{x \in M} \text{dist} (f(x),g(x))$, for $f,g \in \Ho$, or by the uniform distance for homeomorphisms $\delta(f,g) = d(f,g) + d(f^{-1},g^{-1})$. Only the latter is complete, but one can easily check that they span the same topology.

We call $G_{\delta}$ any countable intersection of open subsets of $\Ho$. Since $\Ho$ is a complete metric space, Baire's theorem states that a countable intersection of dense open sets is, in particular, a dense $G_{\delta}$ set. We call \textit{residual} a dense $G_{\delta}$ set, and say that a property is \textit{generic} in $\Ho$ if it is satisfied on at least a residual set. This notion has a nice behaviour under intersection: given a finite (or countable) number of generic properties, the set of homeomorphisms satisfying simultaneously all these properties is still a countable intersection of dense open sets, and is therefore dense. As a consequence, this allows to talk about generic homeomorphisms and list their different properties.

\subsection{The shadowing property}

\begin{definition}
Given $f \in \Ho$ and $\delta > 0$, a \emph{$\delta$-pseudo orbit} $(x_k)_{k \in \Z}$, is a sequence of points in $M$ such that $\text{dist} (f(x_k), x_{k+1}) < \delta$, for all $k \in \Z$. A \emph{$\delta$-periodic pseudo orbit} is a $\delta$-pseudo orbit such that there exists an integer $N > 0$ such that $x_{k+N}=x_k$, for all $k \in \Z$.
\end{definition}

It is natural to think of a $\delta$-pseudo orbit in terms of a roundoff error that a computer would generate when trying to compute the iterations of the point $x_0$ under the transformation $f$.

\begin{definition}[Shadowing property]\label{def:shadow}
We say that $f \in \Ho$ satisfies the:
\begin{itemize}
\item \textit{shadowing property}, if for every $\varepsilon > 0$, there exists $\delta > 0$, such that any $\delta$-pseudo orbit $(x_k)_{k \in \Z}$ is $\varepsilon$-shadowed by the real orbit of a point, namely, there exists $x^* \in M$ such that $\text{dist}(f^k(x^*),x_k) < \varepsilon$, for all $k \in \Z$;
\item \textit{periodic shadowing property}, if for every $\varepsilon > 0$, there exists a $\delta > 0$, such that any $\delta$-periodic pseudo orbit $(x_k)_{k \in \Z}$ is $\varepsilon$-shadowed by the real orbit of a periodic point $x^*$, with same period as $(x_k)_{k \in \Z}$;	
\item \textit{special shadowing property}, if it satisfies both shadowing and periodic shadowing properties.
\end{itemize}
\end{definition}

In other words, if $f$ satisfies the shadowing property, then any pseudo orbit stays close enough to the real orbit of some point $x^*$ of $M$, \textit{a priori} different from $x_0$. \\

\subsection{Main result}

As far as conservative homeomorphisms are concerned, their generic properties are now quite well understood, thanks to some fundamental results obtained in the 1940's by J. Oxtoby and S. Ulam, in the 1960's by A. Katok and A. Stepin and in the 1970's by S. Alpern, V. S. Prasad and P. Lax. In \cite{pa}, the first author provides a general overview of the generic properties of conservative homeomorphisms, be they of topological or ergodic nature (see also the historical survey \cite{ChoP} of  J. Choksi and V. Prasad).

However, it seems that the genericity of the shadowing property is still missing in the conservative case. This paper aims to fill in this gap by showing the following result:

\begin{theorem}
\label{th:shadowing}
The special shadowing property is generic in $\Ho$.
\end{theorem}

Our proof relies on the use of Oxtoby-Ulam's theorem, which provides an adequate subdivision of the manifold $M$ in a very general case, namely we only require $M$ to be a topological manifold. As a consequence, this article also establishes the genericity of the special shadowing property in the dissipative case for the class of topological manifolds.

One must admit that the genericity in the $C^0$ topology can be considered as irrelevant from a practical viewpoint: a generic homeomorphism is nowhere differentiable, and exhibits wild behaviours such as Cantor sets of periodic points of a given period; thus generic homeomorphisms represent badly most of real-world systems. However, results are usually much easier to obtain than in more regular topologies and can constitute a first step for the studies of genericity in $C^r$ topologies, for greater numbers $r$.\\

The shadowing property was first introduced in the works of D. Anosov and R. Bowen, who proved independently that in a neighbourhood of a hyperbolic set, a diffeomorphism has the shadowing property. This result is known as the \textit{shadowing lemma} (see \cite[Paragraph 2.7]{Yocc} or \cite[Theorem 18.1.2]{kh} for a proof). As a consequence, examples of dynamical systems satisfying the shadowing property are provided by Anosov diffeomorphisms. For further details on the notion of shadowing, we refer the reader to the books  \cite{syp} of S. Y. Pilyugin's and \cite{kp} of K. Palmer.

The first proof of genericity of the shadowing property was obtained by K. Yano (see \cite{ky}) in the case $M = \mathbb{S}^1$. Then, using the possibility to approximate any homeomorphism by a diffeomorphism in dimension $n \leq 3$, K. Odani obtained in \cite{ko} the genericity of the shadowing property for manifolds of dimension less than $3$. S. Y. Pilyugin and O. B. Plamenevskaya were able to improve this result in \cite{pp} to any dimension in the case of smooth manifolds. In 2005, P. Koscielniak established in \cite{piotr} the genericity of the shadowing property for homeomorphisms on a compact manifold which possesses a triangulation (smooth manifolds or topological manifolds of dimension $\leq 3$ for example) or a handle decomposition (smooth manifolds or manifolds of dimension $\geq 6$ for example). To the best of our knowledge, this was the best result obtained so far. Notice that our global strategy of proof is similar that of \cite{piotr}.\\

Let us mention some consequences of our main theorem in the conservative case. One knows that a generic conservative homeomorphism is topologically mixing (see \cite{pa}). Now, in \cite[Proposition 23.20]{dgs}, M. Denker, C. Grillenberger and K. Sigmund prove that a homeomorphism which is topologically mixing and satisfies the shadowing property has the specification property\footnote{Morally, a homeomorphism has the periodic specification property if any finite number of pieces of orbits which are sufficiently time-spaced can be shadowed by the real orbit of a periodic point. For a rigourous definition, see \cite{dgs}.}. A very slight adaptation of their proof, using the \textit{periodic} shadowing property instead of the shadowing property, gives the \textit{periodic} specification property. Also, in \cite[Theorem 3.8]{kko}, M. Kulczycki, D. Kwietniak and P. Oprocha prove that the average shadowing property and the asymptotic average shadowing property\footnote{The definition of average shadowing is similar to that of shadowing but allows large deviations in the distance between $f(x_k)$ and $x_{k+1}$, as long as they are rare enough and balanced by a number of very small deviations. For a precise definition, see \cite{kko}.} are satisfied by a topologically mixing system possessing the shadowing property. To sum up, we obtain the following corollary:

\begin{corollary}
\label{cor}
A generic element in $\mathcal H(M,\mu)$ satisfies the specification property, the average shadowing property and the asymptotic average shadowing property.
\end{corollary}

\subsection{Reading guide}

As a reading guide, let us sketch the proof of Theorem \ref{th:shadowing}.

To begin with, we will apply Oxtoby-Ulam-Brown Theorem (Theorem \ref{th:oub}), that will reduce the study to the case where the phase space is the unit cube. That will allow us to define \emph{dyadic subdivisions} (see Definition \ref{def:subdivision}) on our manifold. Then, given a generic homeomorphism $f\in\Ho$, and denoting by $C_i$ the cubes of some fine enough dyadic subdivision, we will prove that:
\begin{enumerate}
\item Each time $f(C_i)\cap C_j\neq \emptyset$, the set $f(C_i)\cap C_j$ has nonempty interior. This will be obtained easily by a ``transversality'' result (first part of the proof of Lemma \ref{lem:od}).
\item Each time $f(C_i)\cap C_j\neq \emptyset$, there exists some small cubes $c_i\subset C_i$ and $c_j\subset C_j$ such that $c_i$ and $c_j$ have a Markovian intersection (see Definition \ref{def:markov}). Homeomorphisms satisfying this property will be called \emph{chained} (see Definition \ref{def:c}). This is the most delicate part of the proof, which will be obtained by creating transverse intersections for some foliations on the cube (see Definition \ref{prop:fol}).
\end{enumerate}
The property of \emph{periodic} shadowing will be deduced from the previous construction by applying a fixed point lemma (Lemma \ref{lem:fp}, due to \cite{zg}).


Section 2 of the present article is a toolbox of key technical lemmas which we will use in Section, 3 in order to prove our main theorem. Eventually, we formulate a few remarks about the proof in Section 4.

\subsection*{Acknowledgements:} The first author is founded by an IMPA/CAPES grant. The second author wishes to acknowledge Alexander Arbieto for his advice, and the members of the Instituto de Matemática of the Universidade Fereral do Rio de Janeiro for their welcome.

\section{Toolbox}

In this section, we present the main technical results which will be used in the proof of main theorem. We begin by presenting the two perturbation results for homeomorphsims we will use throughout this paper.

\subsection{Perturbation lemmas in topology $C^0$}

The two following perturbation lemmas are among the key technical results used to prove the genericity of topological properties in $\Ho$.

\begin{lemma}[Extension of finite maps]
\label{lem:ext}
Let $x_1, \dots,x_n$ be $n$ different points of $M \setminus \partial M$ and $\Phi : \left\{ x_1, \dots, x_n \right\} \rightarrow M$ be an injective map such that $d(\Phi, Id) < \delta$. Then, there exists $\varphi \in \Ho$ such that $\varphi(x_i) = \Phi(x_i)$, for all $i \in \left\{ x_1, \dots, x_n \right\} $ and $d(\varphi, Id) < \delta$. Moreover, given $n$ injective continuous paths $\gamma_i$ joining $x_i$ to $\Phi(x_i)$, the support of $\varphi$ can be chosen in any neighbourhood of the union of the paths $\gamma_i$.
\end{lemma}

\noindent
The proof of this lemma is rather easy but a bit technical. For a detailed proof, see \cite{ou} or \cite[Chapter 2]{pa}.

Before stating the second lemma, we recall the definition of a \textit{bicollared} embedding, which avoids pathological embeddings such as the Alexander horned sphere.

\begin{definition}[Bicollared embeddings]
An embedding $i$ of a manifold $\Sigma$ into an manifold $M$ is said to be \textit{bicollared} if there exists an embedding $j : [-1,1] \times \Sigma \rightarrow M$ such that $j_{\left\{0\right\} \times \Sigma} = i$.
\end{definition}

\begin{lemma}[Local modification]
\label{lem:modif}
Let $\sigma_1,\sigma_2,\tau_1,\tau_2$ be four bicollared embeddings of $\mathbb{S}^{n-1}$ in $\R^n$, such that $\sigma_1$ is in the bounded connected component of $\sigma_2$ and $\tau_1$ in the bounded connected component of $\tau_2$. Let $A_1$ be the bounded connected component of $\R^n - \sigma_1$ and $B_1$ be the bounded connected component of $\R^n - \tau_1$, $\Sigma$ the connected component of $\R^n - \left(\sigma_1 \cup \sigma_2 \right)$ with boundaries $\sigma_1 \cup \sigma_2$ and $\Lambda$ the connected component of $R^n - \left(\tau_1 \cup \tau_2 \right)$ with boundaries $\tau_1 \cup \tau_2$, $A_2$ the unbounded connected component of $R^n - \sigma_2$ and $B_2$ the unbounded connected component of $R^n - \tau_2$.

Consider two homeomorphisms $f_i : A_i \rightarrow B_i$, such that they both preserve or reverse the orientation. Then, there exists a homeomorphism $f : \R^n \rightarrow \R^n$ such that $f = f_1$ on $A_1$ and $f=f_2$ on $A_2$.

Moreover, if we assume $\lambda(A_1) = \lambda(B_1)$ and $\lambda(\Sigma) = \lambda(\Lambda)$ and the homeomorphisms $f_i$ conservative, then $f$ can be chosen conservative too.
\end{lemma}

\begin{figure}[h!]
\begin{center}

\includegraphics[scale=0.4]{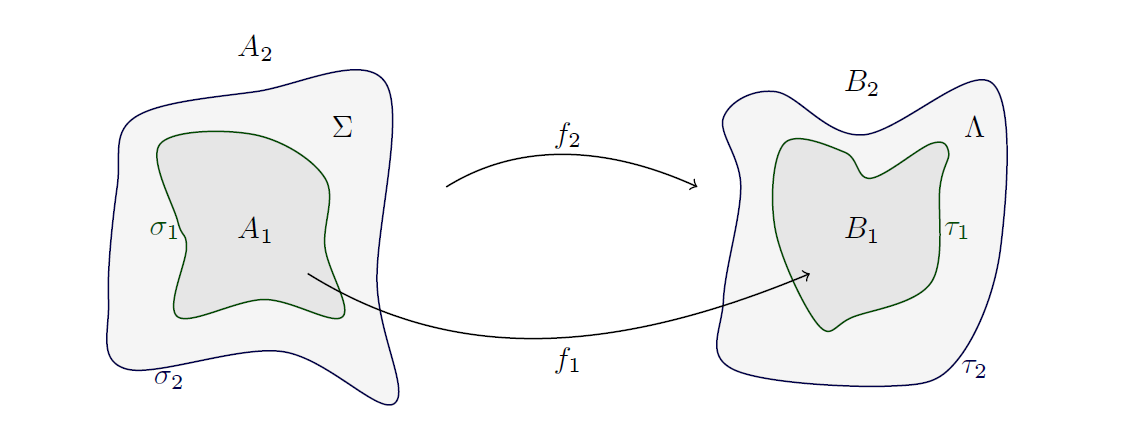} 
\caption{The local modification theorem}

\end{center}
\end{figure}

Although the statement of local modification lemma is quite natural, its proof involves the very hard annulus theorem. For a proof, see \cite{df} or \cite[Chapter 3]{pa}.

\subsection{Dyadic subdivision}

$\lambda$ will denote the Lebesgue measure on $\R^n$ and $I^n = [0,1]^n$, the unit cube in $\R^n$. Recall that $M$ is a compact connected manifold with or without boundary, of dimension $n \geq 2$, endowed with a good measure $\mu$. One of the most fundamental results in the theory of generic (conservative) homeomorphisms is a combination of theorems of Brown (see \cite{brown}) and of Oxtoby-Ulam (see \cite{ou}). A detailed proof of this result can be found in \cite{ap}, Appendix 2:

\begin{theorem}[Oxtoby-Ulam-Brown]
\label{th:oub}
Let $\mu$ be a good Borel probability measure on $M$. Then, there exists $\phi : I^n \rightarrow M$ continuous such that:
\begin{enumerate}
\item $\phi$ is surjective,
\item $\phi|_{\overset{\circ}{I^n}}$ is a homeomorphism on its image,
\item $\phi(\partial I^n )$ is a closed subset of $M$, of empty interior, and disjoint of $\phi(\overset{\circ}{I^n})$,
\item $\mu(\phi(\partial I^n)) = 0$,
\item $\phi_* (\lambda) = \mu$.
\end{enumerate}
\end{theorem}

Now, once for all, we fix such a map $\phi : I^n \rightarrow M$. This allows us to introduce the notion of dyadic subdivision.

\begin{definition}[Dyadic subdivision]\label{def:subdivision}
A \emph{dyadic cube} of order $m$ of $M$ is the image in $M$ by $\phi$ of a cube $\prod_{i=1}^{n} [\dfrac{k_i}{2^m},\dfrac{k_i + 1}{2^m}]$, with $0 \leq k_i \leq 2^m - 1$. The \emph{dyadic subdivision} $\mathcal{D}_m$ of order $m$ of the manifold $M$ is the collection of dyadic cubes of order $m$ of $M$. In the following, $p_m = 2^{nm}$ will denote the number of cubes of the subdivision.
\end{definition}

For a dyadic subdivision $\mathcal{D}_m = (C_i)_{1 \leq i \leq p_m}$, we will denote by $\chi (\mathcal{D}_m)$ the maximum diameter of its cubes, namely:
\[ \chi (\mathcal{D}_m) = \max_{C_i \in \mathcal{D}_m} \text{diam } C_i.\]
These dyadic subdivisions satisfy some good properties:
\begin{itemize}
\item Each cube is connected and obtained as the closure of the open cube $\prod_{i=1}^{n} ]\dfrac{k_i}{2^m},\dfrac{k_i + 1}{2^m}[$.
\item For all $m$, $\mathcal{D}_m$ is a cover of $M$ by a finite number of cubes of same measure and whose interiors pairwise disjoint.
\item For all $m$, $\mathcal{D}_{m+1}$ is a refinement of the subdivision $\mathcal{D}_m$.
\item The measure of the cubes as well as $\chi(\mathcal{D}_m)$ tend to $0$ as $m \rightarrow \infty$. The measure of the boundary of each cube is zero.
\end{itemize}
Note that the image of the dyadic subdivision by any $f \in \Ho$ is still a dyadic subdivision satisfying the same properties.

\subsection{Markovian intersections}

The mechanism of perturbation which will produce shadowing property is based on the notion of chained homeomorphism (Definition \ref{def:c}, itself based on \textit{Markovian intersections}).

\begin{definition}
\label{def:mark}
We call \textit{rectangle} a subset $R \subset M$ such that $R = \phi(I^n)$, where $\phi: I^n\to \phi(I^n)\subset M$ is a homeomorphism. We call \textit{faces} of $R$ the image by $\phi$ of the faces\footnote{By definition, a face of $I^n$ is one of the $n-1$-dimensional cubes constituting the boundary of $I^n$.} of $I^n$. We call \textit{horizontal} the faces $R^- = \phi(I^{n-1} \times \left\{ 0 \right\})$ and $R^+ = \phi(I^{n-1} \times \left\{ 1 \right\})$ and \textit{vertical} the others. We say that a rectangle $R' \subset R$ is a \emph{strict horizontal} (resp. vertical) subrectangle of $R$ if the horizontal (resp. vertical) faces of $R'$ are strictly disjoint from those of $R$ and the vertical (resp. horizontal) faces of $R'$ are included in those of $R$.
\end{definition}

\noindent
Given $x \in \R^n$, we will denote by $\pi_1(x)$ its first coordinate. Following P. Zgliczynski and M. Gidea's article \cite{zg}, we define Markovian intersections in the following way:

\begin{definition}\label{def:markov}
Let $f$ be a homeomorphism of $M$, $R_1$ and $R_2$ two rectangles of $M$. We say that $f(R_1) \cap R_2$ is a \textit{Markovian intersection} if there exists an horizontal subrectangle $H$ of $R_1$ and a homeomorphism $\phi$ from a neighbourhood of $H\cup R_2$ to $\R^n$ such that:
\begin{itemize}
\item $\phi(R_2) = [-1,1]^n$;
\item either $\phi(f(H^+)) \subset \left\{x\mid \pi_1(x) > 1 \right\}$ and $\phi(f(H^-)) \subset \left\{x\mid \pi_1(x) < -1 \right\}$, or $\phi(f(H^-)) \subset \left\{x\mid \pi_1(x) > 1 \right\}$ and $\phi(f(H^+)) \subset \left\{x\mid \pi_1(x) < -1 \right\}$;
\item $\phi(H) \subset \left\{x\mid \pi_1(x) < -1 \right\} \cup [-1,1]^n \cup \left\{x\mid \pi_1(x) > 1 \right\}$.
\end{itemize}
\end{definition}

\begin{figure}[h!]
\begin{center}

\includegraphics[scale=0.8]{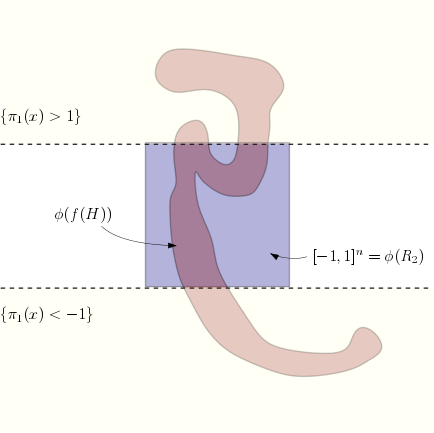} 
\caption{A Markovian intersection}

\end{center}
\end{figure}


\noindent
The following two results show that the markovian intersections have nice behaviours under $\mathcal{C}^0$ perturbation and iteration. Proofs of them can be obtained as a combination of Theorem 16, Theorem 13 and Corollary 12 of \cite{zg}.

\begin{proposition}
\label{prop:mark}
A Markovian intersection is $\mathcal{C}^0$ robust, namely if the intersection $f(R_1) \cap R_2$ is Markovian, then it is still true in a $\mathcal{C}^0$ neighbourhood of $f$.
\end{proposition}

\begin{proposition}
\label{prop:mark2}
Given three rectangles $R_1$, $R_2$ and $R_3$, if the intersections $f(R_1) \cap R_2$ and $f(R_2) \cap R_3$ are Markovian, then the intersection $f^2(R_1) \cap R_3$ is Markovian too.
\end{proposition}

\begin{definition}[$m$-chained homeomorphism]
\label{def:c}

Given a homeomorphism $f$, we say that $f$ is \textit{$m$-chained} if for every cube $C_i$ of $\Df_m$ there exists a rectangle $c_i\subset C_i$, such that for every $i,j$ such that $f(C_i)\cap C_j \neq \emptyset$, the intersection $f(c_i)\cap c_j$ is nonempty and Markovian.
\end{definition}

\noindent
We will use the following lemma to obtain periodic points for the periodic shadowing property. It is a simplified version of \cite[Theorems 4, 16]{zg}. The proof is based on an argument of homotopy and theory of degree:

\begin{lemma}
\label{lem:fp}
Let $f$ be a homeomorphism and $R$ be a rectangle such that $f(R) \cap R$ is Markovian. Then, there exists a fixed point for $f$ in $R$.
\end{lemma}

\subsection{Foliations}

In this paragraph, we recall an elementary result on foliations.

\begin{definition}[Transverse intersection]
Given two foliations $\mathcal{F}, \mathcal{F}'$ such that $\text{dim}(\mathcal{F}) + \text{dim}(\mathcal{F}') = n$, we will say that two leaves $L$ of $\mathcal{F}$ and $L'$ of $\mathcal{F}'$ \emph{intersect transversally}, and denote by $L \pitchfork L'$, if either their intersection is empty, or there exists an open set $\ell$ of the leaf $L$ such that $\text{Card}(\ell \cap L') = 1$. We will say that the two foliations $\mathcal{F}$ and $\mathcal{F}'$ \emph{intersect transversally} it they have two leaves that intersect transversally.
\end{definition}

\noindent
In the following, we will use this definition with $\text{dim}(\mathcal{F}) = \text{codim}(\mathcal{F}') = 1$.

\begin{proposition}
\label{prop:fol}
Assume $f$ is $m$-nice and consider two cubes $C_i,C_j$ of the subdivision $\Df_m$ such that $f(C_i) \cap C_j \neq \emptyset$. We consider a smooth foliation $\Ff$ of $C_i$ and a smooth foliation $\Ff'$ of $C_j$ such that $\text{dim}(\Ff) + \text{dim}(\Ff') = n$. If $f(\Ff)$ does not intersect transversally $\Ff'$ on $\overset{\circ}{C_j}$, then there exists a $\Cf^0$ perturbation of $f$, as small as desired, such that this intersection is transverse. 
\end{proposition}

\begin{proof}
We assume that there does not exist any transverse intersection of leaves in $\Uf = f(\overset{\circ}{C}_i) \cap \overset{\circ}{C}_j$. Consider a point $y \in \Uf$ and the leaf $L'$ of $\Ff'$ passing through $y$. We denote by $(f_1,\dots,f_k)$ an orthonormal basis of $T_y L'$ and we complete it into an orthonormal basis $(f_1,\dots,f_n)$ of $\R^n$. Now, we consider the point $x = f^{-1}(y)$, the leaf $L$ of $\Ff$ passing through $x$ and an orthonormal basis $(e_1,\dots,e_{n-k})$ of $T_x L$ which we complete in a basis $(e_1,\dots,e_n)$ of $\R^n$. We denote by $A \in O(n)$ the linear orthogonal application such that $Ae_i = f_{k+i}, i \in \{ 1\dots ,n \}$, where the index is taken modulo $n$. Then, one can apply Lemma \ref{lem:modif} to replace locally around $y$ the homeomorpihsm $f$ by the affine volume-preserving transformation taking value $x$ on $y$ and of linear part $A$. This gives a homeomorphism of $\Ho$ as close as wanted to $f$, for which the intersection is transverse.
\end{proof}

\section{Proof of main theorem}

\noindent
In this section, we prove Theorem \ref{th:shadowing}. Let us define:
\[ A_{\varepsilon} = \big\{ f \in \Ho \mid f \text{ is } m\text{-chained for some } \mathcal{D}_m \text{ with } \chi(\mathcal{D}_m) < \varepsilon \big\} \]
The proof of the theorem immediately follows from these two lemmas:

\begin{lemma}
\label{lem:od}
For any $\varepsilon>0$, the set $A_{\varepsilon}$ is open and dense in $\Ho$.
\end{lemma}

\begin{lemma}
\label{lem:sp}
If $f \in \cap_{p \in \N} A_{1/p}$, then $f$ satisfies the special shadowing property. 
\end{lemma}

\begin{proof}[Proof of Lemma \ref{lem:od}] 
The fact that $A_\varepsilon$ is open easily follows from Proposition \ref{prop:mark}. Therefore, we only have to prove that this set is dense in $\Ho$.

We fix $\varepsilon > 0$, $f \in \Ho$ and $\kappa > 0$. We want to show that there exists $g \in A_{\varepsilon}$ such that $d(f,g) < \kappa$. We consider a dyadic subdivision $\mathcal{D}_m$ such that $\chi(\mathcal{D}_m) < \min(\varepsilon, \kappa, \omega(\kappa))$, where $\omega(\kappa)$ denotes the modulus of uniform continuity of $f$. Our goal is to create Markovian intersections between each pair of cubes $C_i$ and $C_j$ such that $f(C_i) \cap C_j \neq \emptyset$.\\

Firstly, we prove that making a small perturbation of $f$ if necessary, each time $f(C_i) \cap C_j \neq \emptyset$, the intersection has nonempty interior.

Assume that $f(C_i) \cap C_j \neq \emptyset$ but $f(\overset{\circ}{C_i}) \cap \overset{\circ}{C_j} = \emptyset$. This means that $f(\partial C_i) \cap \partial C_j \neq \emptyset$; consider $f(x)=y$ in this intersection. Now, by the extension of finite maps (Lemma \ref{lem:ext}), one can find a homeomorphism $\varphi \in \Ho$ with support\footnote{The support of a homeomorphism $\psi$ is defined as the closure of the largest set $K$ such that $\psi|_K \neq Id$.} in a neighbourhood of $y$, such that $\tilde{f} = \varphi \circ f$ is as close to $f$ as wanted, and $\tilde{f}(x) \in \overset{\circ}{C_j}$. Therefore $\tilde{f}(\overset{\circ}{C_i}) \cap \overset{\circ}{C_j} \neq \emptyset$ (see Figure \ref{fig:grille}).\\

\begin{figure}[h]
\begin{center}

\includegraphics[scale=.6]{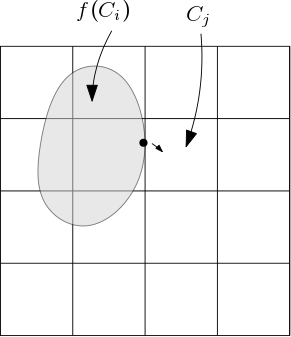} 
\caption{Local modification of $f$, seen in the cube $I^2$}\label{fig:grille}

\end{center}
\end{figure}

Now, on each cube $C_i\in\Df_m$, we consider the foliation $\Ff$ by vertical lines and the foliation $\Hf$ by horizontal hyperplanes. For each intersection $f(C_i)\cap C_j\neq\emptyset$, we look at both foliations $f(\Ff \cap C_i)$ and $\Hf \cap C_j$. By Proposition \ref{prop:fol}, we can always assume that, up to a small perturbation of $f$, this intersection is transverse. We denote by $L_i^j$ and ${L'}_i^j$ the leaves of $\Ff \cap C_i$ and $\Hf \cap C_j$, such that $f(L_i^j)$ intersects transversally ${L'}_i^j$. 

On each cube $C_i$, we consider a smooth path $\gamma_i$ such that for every transverse intersection $L_i^j \pitchfork f^{-1}({L'}_i^j)\neq\emptyset$, the path $\gamma_i$ coincides with the leaf $L_i^j$ on a small neighbourhood of $L_i^j \pitchfork f^{-1}({L'}_i^j)$. Looking at the cubes $C_k$ such that $f(C_k) \cap C_i \neq \emptyset$, we also consider a smooth codimension 1 submanifold $\sigma_i$ such that for every transverse intersection $f(L_k^i) \pitchfork {L'}_k^i\neq\emptyset$, $\sigma_i$ coincides with the leaf ${L'}_k^i$ on a small neighbourhood of the transverse intersection $f(L_k^i) \pitchfork {L'}_k^i$. Remark that by construction, for every $i$ and $j$, we have the transverse intersection $f(\gamma_i) \pitchfork \sigma_j$.

Then, we consider a $\delta$ tubular neighbourhood $\Gamma_i$ of the path $\gamma_i$ and a $\delta'$ tubular neighbourhood $\Sigma_i$ of the submanifold $\sigma_i$, as well as two (conservative) homeomorphisms $\phi_i$ and $\phi_i'$ of the cube $C_i$, such that (see Figures \ref{fig:markov1} and \ref{fig:markov2}):
\begin{itemize}
\item $\Gamma_i$ and $\Sigma_i$ have the same volume,
\item $\phi_i$ and $\phi_i'$ have support in $C_i$,
\item if we denote $c_i$ the cube with same centre as $C_i$ and same volume as $\Gamma_i$, we have $\phi_i(c_i) = \Gamma_i$ and $\phi'_i(c_i) = \Sigma_i$;
\item the image of the vertical (resp. horizontal) faces of the cube $c_i$ by $\phi_i$ (resp. $\phi'_i$) is contained in a small neighbourhood of the boundary of $\gamma_i$ (resp. $\sigma_i$).
\end{itemize}

\begin{figure}[h]
\begin{center}

\includegraphics[scale=0.8]{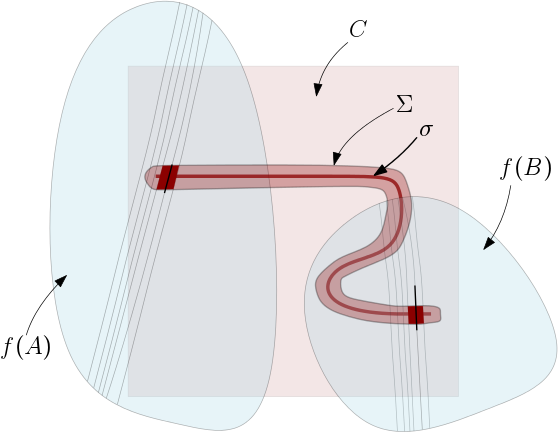} 
\caption{Cubes $A$ and $B$ have their image intersecting the cube $C$. The vertical lines represent the image by $f$ of the vertical foliation of cubes $A$ and $B$. $\Sigma$ is transverse to the image of the foliation by $f$ in the dark red areas.}\label{fig:markov1}

\end{center}
\end{figure}

\begin{figure}[h]
\begin{center}

\includegraphics[scale=0.6]{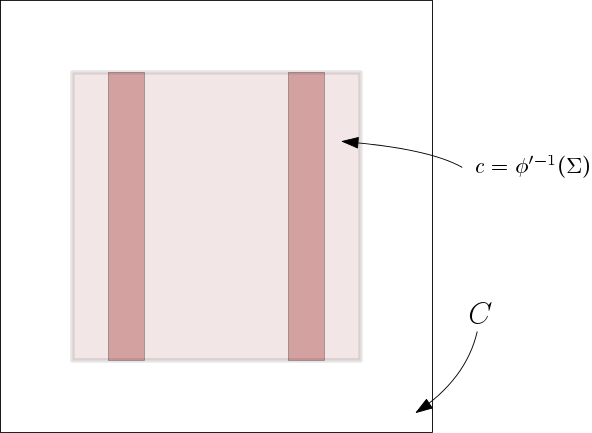} 
\caption{The application $\phi'$ unfolds the cube $c$ in the manifold $\Sigma$. The previous dark red areas in Figure \ref{fig:markov1} are the image by $\phi'$ of the dark red areas above.}\label{fig:markov2}

\end{center}
\end{figure}

As $\gamma_i$ and $\sigma_i$ are smooth, for $\delta$ (and thus $\delta'$) small enough, for any $i,j$ such that $f(C_i)\cap C_j\neq\emptyset$, the rectangles $f(\phi_i(c_i))$ and $\phi'_j(c_j)$ have a Markovian intersection. Then, we set:
\[g = \big(\prod_j (\phi'_j)^{-1} \big) f \big(\prod_i \phi_i\big),\]
By the inequality $\chi(\mathcal{D}_m) < \min(\varepsilon, \kappa, \omega(\kappa))$, and since the $\phi_i$ and $\phi'_i$ have their support included in a single cube of the subdivision, one gets a homeomorphism which is $\kappa$-close to $f$ and which belongs to $A_{\varepsilon}$.

\end{proof}

\begin{proof}[Proof of Lemma \ref{lem:sp}]
We now prove that if $f \in \cap_{p \in \N} A_{1/p}$, then $f$ satisfies the special shadowing property. 

Let us consider $f \in A_{\varepsilon}$. This gives us a subdivision $\Df_m$ such that $f$ is $m$-chained for $\Df_m$ and $\chi(\Df_m) < \varepsilon$. We set:
\[\delta < \min_{f(C_i) \cap C_j = \emptyset} \text{dist } (f(C_i),C_j).\]
Note that the set of index on which the minimum is taken may be empty: in that case, any $\delta > 0$ works.



Now, consider any $\delta$-pseudo orbit $(y_k)_{k \in \Z}$ of $f$. As the cubes $C_i$ form a partition of $M$, one can chose a sequence $(i_k)_{k \in \Z}$ of indices such that $y_k \in C_{i_k}$ for any $k\in\Z$. Therefore, $f(C_{i_k}) \cap C_{i_{k+1}} \neq \emptyset$, otherwise we would have $d(f(C_{i_{k}}),C_{i_{k+1}}) > \delta$, which is impossible because $(y_k)_{k \in \Z}$ is a $\delta$-pseudo orbit.

Recall that by definition of $A_{\varepsilon}$, we have associated a small rectangle $c_i$ to any cube $C_i$ of $\mathcal D_m$, such that for any $i,j$ such that $f(C_i)\cap C_j\neq\emptyset$, the intersection $f(c_i) \cap c_j$ is nonempty and Markovian. Then, by an easy recurrence on $n$, using Lemma \ref{prop:mark2}, we obtain that for any $n$, the intersection $f^n(c_{i_{-n}}) \cap f^{-n}c_{i_n}$ is nonempty and Markovian. Let $x_n$ be a point of this intersection. Then, for any $k\in\{-n,\dots,n\}$, as $f^{k}(x_n)\in c_{i_k}\subset C_{i_k}\ni y_k$, one has
\[ d(f^{k}(x_n), y_k) < \chi(\mathcal{D}_m) < \varepsilon.\]
Therefore, $x_n$ $\varepsilon$-traces the finite $\delta$-pseudo orbit $(y_k)_{-n \leq k \leq n}$. Since $M$ is a compact manifold, a subsequence of $(x_n)_{n > 0}$ converges towards some point $x$, which $\varepsilon$-traces the $\delta$ pseudo-orbit $(y_k)_{k \in \Z}$. 

In the case of a periodic $\delta$-pseudo orbit, the shadowing by the real orbit of a periodic point follows immediately from the previous reasoning and Lemma \ref{lem:fp} which allows to exhibit a periodic orbit from a periodic chain of cubes.

Eventually, this shows that the $G_{\delta}$ set $\cap_{p \in \N} A_{1/p}$ is contained in the set of homeomorphisms of $\Ho$ satisfying the special shadowing property.

\end{proof}

\section{Remarks}

As a conclusion, we formulate a few remarks:
\begin{itemize}
\item As mentioned briefly in the introduction, the proof actually still holds in the non-conservative case and provides an alternative proof to \cite{piotr}. Indeed, topologically, the only difference is that the image of a cube may be strictly contained in another cube but this fact does not have any consequence on our proof.
\item The non-shadowing property also holds on a dense set in $\Ho$. Indeed, in \cite{pa} (pages 33-34), the density of the maps $f \in \Ho$ which have an iterate equal to the identity on an open set ($f^p = Id$ on an open set $V \subset M$, for some $p > 0$) is proved. This property contradicts immediately the shadowing property.
\item Given a $\varepsilon > 0$, the proof also provides an upper-bound for the $\delta > 0$ that can be chosen. If $\mathcal{D}_m$ is a subdivision of diameter less than $\varepsilon$, then on can take:
\[ \delta < \min_{f(C_i) \cap C_j = \emptyset} \text{dist } (f(C_i), C_j) \]
\item The specification property does not hold generically in the dissipative case. Indeed, it is known that the specification property implies topological mixing (see \cite{dgs}, Proposition 21.3), which, in turn, does not hold on an open set in the dissipative case.
\end{itemize}

\bibliographystyle{amsplain}

\end{document}